\documentclass{amsart}
\usepackage{amsmath}
\usepackage{amssymb}
\usepackage{amsfonts}

\newtheorem{theorem}{Theorem}
\theoremstyle{plain}

\newtheorem{definition}{Definition}

\newtheorem{proposition}{Proposition}
\newtheorem{remark}{Remark}

\numberwithin{equation}{section}
\input{tcilatex}

\begin{document}
\title{Fuzzy soft seperation axioms with sense of Ganguly and Saha}
\author{S. ATMACA}
\address{Cumhuriyet University, Faculty of Science, Department of
Mathematics, 58140 Sivas / Turkey}
\email{seatmaca@cumhuriyet.edu.tr}

\begin{abstract}
Tanay and Kandemir \cite{TK} introduced the topological structure of fuzzy
soft sets. In 2013, Manatha and Das \cite{md} defined seperation axioms on
fuzzy soft topological spaces. In this paper, we generalized form of the
seperation axioms.using fuzzy soft quasi-coincidence with sense of Ganguly
and Saha \cite{GS}. By using this notions, we also give some basic theorems
of seperation axioms in classical topological spaces.
\end{abstract}

\maketitle

\section{\protect\bigskip Introduction}

Most problems in engineering, medical scinece, economics and environment
etc. have uncertienties. Several set theories have been given in order to
mathematically model these uncertainties. Soft sets and fuzzy sets are
leading two of these theories.

Recently, different theories combined for modelling the problems, inculude
these uncertainties, more efficiently. One of these is the fuzzy soft sets
which is combination of fuzzy sets and soft sets and is defined by Maji et
al. \cite{MB3}. This theory, is interested in researchers in a short time,
has been applied several directions, such as topology \cite{Az, Bh, TK, Ts,
rs}, algebraic structures \cite{AA, io} and especially decision making \cite%
{fj, zk, arr, x2}.

In 2011, Tanay and Kandemir \cite{TK} defined the topological structure of
fuzzy soft sets. To improve this concept many researchers studied on this
field. Mahanta and Das \cite{md} mentioned seperation axioms in fuzzy soft
topological spaces.

In the present paper, we defined new seperation axioms, which are more
general than seperation axioms of Mahanta and Das, by the sense of Ganguly
and Saha \cite{GS} in fuzzy soft topological spaces. Using these
definitions, we gave the theorems which are important for seperation axioms
and taking place in classical topological spaces.

\section{Preliminaries}

Throughout this paper $X$ denotes initial universe, $E$ denotes the set of
all possible parameters which are attributes, characteristic or properties
of the objects in $X$, and the set of all subsets of $X$ will be denoted by $%
P(X)$.

\begin{definition}
\cite{Z} A fuzzy set $A$ of a non-empty set $X$ is characterized by a
membership function $%
%TCIMACRO{\U{3bc} }%
%BeginExpansion
\mu
%EndExpansion
_{A}:X\rightarrow \lbrack 0,1]$ whose value $%
%TCIMACRO{\U{3bc} }%
%BeginExpansion
\mu
%EndExpansion
_{A}(x)$ represents the "grade of membership" of $x$ in $A$ for $x\in X$.
\end{definition}

Let $I^{X}$ denotes the family of all fuzzy sets on $X$. If $A,B\in I^{X}$,
then some basic set operations for fuzzy sets are given by Zadeh as follows:

(1) $A\leq B\Leftrightarrow 
%TCIMACRO{\U{3bc} }%
%BeginExpansion
\mu
%EndExpansion
_{A}(x)\leq 
%TCIMACRO{\U{3bc} }%
%BeginExpansion
\mu
%EndExpansion
_{B}(x)$, for all $x\in X$.

(2) $A=B\Leftrightarrow 
%TCIMACRO{\U{3bc} }%
%BeginExpansion
\mu
%EndExpansion
_{A}(x)=%
%TCIMACRO{\U{3bc} }%
%BeginExpansion
\mu
%EndExpansion
_{B}(x)$, for all $x\in X$.

(3) $C=A\vee B\Leftrightarrow 
%TCIMACRO{\U{3bc} }%
%BeginExpansion
\mu
%EndExpansion
_{C}(x)=%
%TCIMACRO{\U{3bc} }%
%BeginExpansion
\mu
%EndExpansion
_{A}(x)\vee 
%TCIMACRO{\U{3bc} }%
%BeginExpansion
\mu
%EndExpansion
_{B}(x)$, for all $x\in X$.

(4) $D=A\wedge B\Leftrightarrow 
%TCIMACRO{\U{3bc} }%
%BeginExpansion
\mu
%EndExpansion
_{D}(x)=%
%TCIMACRO{\U{3bc} }%
%BeginExpansion
\mu
%EndExpansion
_{A}(x)\wedge 
%TCIMACRO{\U{3bc} }%
%BeginExpansion
\mu
%EndExpansion
_{B}(x)$, for all $x\in X$.

(5) $E=A^{C}\Leftrightarrow 
%TCIMACRO{\U{3bc} }%
%BeginExpansion
\mu
%EndExpansion
_{E}(x)=1-%
%TCIMACRO{\U{3bc} }%
%BeginExpansion
\mu
%EndExpansion
_{A}(x)$, for all $x\in X$.

A fuzzy point in $X$, whose value is $\alpha $ $(0<\alpha \leq 1)$ at the
support $x\in X$, is denoted by $x_{\alpha }$ \cite{Z}. A fuzzy point $%
x_{\alpha }\in A$, where $A$ is a fuzzy set in $X$ iff $\alpha \leq \mu
_{A}(x)$ \cite{Z}. The class all fuzzy points will be denoted by $S(X)$.

\begin{definition}
\cite{PM} For two fuzzy sets $A$ and $B$ in $X$, we write $AqB$ to mean that 
$A$ is quasi-coincident with $B$, i.e., there exists at least one point $%
x\in X$ such that $%
%TCIMACRO{\U{3bc} }%
%BeginExpansion
\mu
%EndExpansion
_{A}(x)+%
%TCIMACRO{\U{3bc} }%
%BeginExpansion
\mu
%EndExpansion
_{B}(x)>1$. If $A$ is not quasi-coincident with $B$, then we write $A%
\overline{q}B$.
\end{definition}

\begin{definition}
\cite{M} Let $X$ be the initial universe set and $E$ be the set of
parameters. A pair $(F,A)$ is called a soft set over $X$ where $F$ is a
mapping given by $F:A\longrightarrow P(X)$ and $A\subseteq E$.
\end{definition}

In the other words, the soft set is a parametrized family of subsets of the
set $X$. Every set $F(e)$, for every $e\in A$, from this family may be
considered as the set of $e$-elements of the soft set $(F,A)$.

\begin{definition}
\cite{MB3}Let $A\subseteq E$. A pair $(f,A)$ is called a fuzzy soft set over 
$X$ if $f:A\longrightarrow I^{X}$ is a function.
\end{definition}

We will use $FS(X,E)$ instead of the family of all fuzzy soft sets over $X$.

Roy and Samanta \cite{rs} did some modifications in above definition
analogously ideas made for soft sets.

\begin{definition}
\cite{rs}Let $A\subseteq E$. A fuzzy soft set $f_{A}$ over universe $X$ is
mapping from the parameter set $E$ to $I^{X}$, i.e., $f_{A}:E\longrightarrow
I^{X}$, where $f_{A}(e)\neq 0_{X}$ if $e\in A\subset E$ and $f_{A}(e)=0_{X}$
if $e\notin A$, where $0_{X}$ denotes empty fuzzy set on $X$.
\end{definition}

\begin{definition}
\cite{rs} The fuzzy soft set $f_{\emptyset }\in FS(X,E)$ is called null
fuzzy soft set, denoted by $\widetilde{0}_{E}$, if for all $e\in E$, $%
f_{\emptyset }(e)=0_{X}$.
\end{definition}

\begin{definition}
\cite{rs}Let $f_{E}\in FS(X,E)$. The fuzzy soft set $f_{E}$ is called
universal fuzzy soft set, denoted by $\widetilde{1}_{E}$, if for all $e\in E$%
, $f_{E}(e)=1_{X}$ where $1_{X}(x)=1$ for all $x\in X$.
\end{definition}

\begin{definition}
\cite{rs}Let $\ f_{A},g_{B}\in FS(X,E)$. $f_{A}$ is called a fuzzy soft
subset of $g_{B}$ if $f_{A}(e)\leq g_{B}(e)$ for every $e\in E$ and we write 
$f_{A}\sqsubseteq g_{B}$.
\end{definition}

\begin{definition}
\cite{rs}Let $f_{A},g_{B}\in FS(X,E)$. $f_{A}$ and $g_{B}$ are said to be
equal, denoted by $f_{A}=g_{B}$ if $f_{A}\sqsubseteq g_{B}$ and $%
g_{B}\sqsubseteq f_{A}$.
\end{definition}

\begin{definition}
\cite{rs}Let $f_{A},g_{B}\in FS(X,E)$. Then the union of $f_{A}$ and $g_{B}$
is also a fuzzy soft set $h_{C}$, defined by $h_{C}(e)=f_{A}(e)\vee g_{B}(e)$
for all $e\in E$, where $C=A\cup B$. Here we write $h_{C}=f_{A}\sqcup g_{B}$.
\end{definition}

\begin{definition}
\cite{rs}Let $f_{A},g_{B}\in FS(X,E)$. Then the intersection of $f_{A}$ and $%
g_{B}$ is also a fuzzy soft set $h_{C}$, defined by $h_{C}(e)=f_{A}(e)\wedge
g_{B}(e)$ for all $e\in E$, where $C=A\cap B$. Here we write $%
h_{C}=f_{A}\sqcap g_{B}$.
\end{definition}

\begin{definition}
\cite{TK}Let $f_{A}\in FS(X,E)$. The complement of $f_{A}$, denoted by $%
f_{A}^{c}$, is a fuzzy soft set defined by $f_{A}^{c}(e)=1-f_{A}(e)$ for
every $e\in E$.

Let us call $f_{A}^{c}$ to be fuzzy soft complement function of $f_{A}$.
Clearly $(f_{A}^{c})^{c}=f_{A}$, $(\widetilde{1}_{E})^{c}=\widetilde{0}_{E}$
and $(\widetilde{0}_{E})^{c}=\widetilde{1}_{E}$.
\end{definition}

\begin{definition}
\cite{KA1}Let $FS(X,E)$ and $FS(Y,K)$ be the families of all fuzzy soft sets
over $X$ and $Y$, respectively. Let $u:X\rightarrow Y$ and $p:E\rightarrow K$
be two functions. Then $f_{up}$ is called a fuzzy soft mapping from $X$ to $%
Y $ and denoted by $f_{up}:FS(X,E)\rightarrow FS(Y,K)$.

(1) Let $f_{A}\in FS(X,E)$, then the image of $f_{A}$ under the fuzzy soft
mapping $f_{up}$ is the fuzzy soft set over $Y$ defined by $f_{up}(f_{A})$,
where

$f_{up}(f_{A})(k)(y)=\left\{ 
\begin{array}{cc}
{\Large \vee }_{x\in u^{-1}(y)}({\Large \vee }_{e\in p^{-1}(k)\cap
A}f_{A}(e))(x) & \text{if }u^{-1}(y)\neq \varnothing \text{, }p^{-1}(k)\cap
A\neq \varnothing \text{;} \\ 
0_{Y} & \text{otherwise.}%
\end{array}%
\right. $

(2) Let $g_{B}\in FS(Y,K)$, then the preimage of $g_{B}$ under the fuzzy
soft mapping $f_{up}$ is the fuzzy soft set over $X$ defined by $%
f_{up}^{-1}(g_{B})$, where

$f_{up}^{-1}(g_{B})(e)(x)=\left\{ 
\begin{array}{cc}
g_{B}(p(e))(u(x)) & \text{for }p(e)\in B\text{;} \\ 
0_{X} & \text{otherwise.}%
\end{array}%
\right. $
\end{definition}

If $u$ and $p$ are injective then the fuzzy soft mapping $f_{up}$ is said to
be injective. If $u$ and $p$ are surjective then the fuzzy soft mapping $%
f_{up}$ is said to be surjective. The fuzzy soft mapping $f_{up}$ is called
constant, if $u$ and $p$ are constant.

\begin{theorem}
\cite{KA1}Let $f_{A},(f_{A_{i}})\in FS(X,E)$ and $g_{B},(g_{B_{i}})\in
FS(Y,K)$, $i\in J$, where $J$ is an index set.

(1) If $(f_{A_{1}})\sqsubseteq (f_{A_{2}})$, then $f_{up}(f_{A_{1}})%
\sqsubseteq f_{up}(f_{A_{2}})$.

(2) If $(g_{B_{1}})\sqsubseteq (g_{B_{2}})$, then $f_{up}^{-1}(g_{B_{1}})%
\sqsubseteq f_{up}^{-1}(g_{B_{2}})$.
\end{theorem}

\begin{definition}
\cite{md}A fuzzy soft set $g_{A}$ denoted by $e_{g_{A}}$, if for the element 
$e\in A$, $g(e)\neq 0_{X},$ and $g(e^{\prime })=0_{X}$, $\forall e^{\prime
}\in A-\{e\}$.
\end{definition}

\begin{definition}
\cite{Az}The fuzzy soft set $f_{A}\in FS(X,E)$ is called fuzzy soft point if 
$A=\{e\}\subseteq E$ and $f_{A}(e)$ is a fuzzy point in $X$ i.e. there
exists $x\in X$ such that $f_{A}(e)(x)=\alpha $ ($0<\alpha \leq 1$) and $%
f_{A}(e)(y)=0$\ for all $y\in X-\{x\}$. We denote this fuzzy soft point $%
f_{A}=e_{x}^{\alpha }=\{(e,x_{\alpha })\}$.
\end{definition}

\begin{definition}
\cite{Az}Let $e_{x}^{\alpha }$, $f_{A}\in FS(X,E)$. We say that $%
e_{x}^{\alpha }\widetilde{\in }f_{A}$ read as $e_{x}^{\alpha }$ belongs to
the fuzzy soft set $f_{A}$ if for the element $e\in A$, $\alpha \leq
f_{A}(e)(x)$.
\end{definition}

Evidently, every fuzzy soft set $f_{A}$ can be expresssed as the union of
all the fuzzy soft points which belong to $f_{A}$.

\begin{definition}
\cite{Az}Let $f_{A}$, $g_{B}\in FS(X,E)$. $f_{A}$ is said to be soft
quasi-coincident with $g_{B}$, denoted by $f_{A}qg_{B}$, if there exist $%
e\in E$ and $x\in X$ such that $f_{A}(e)(x)+g_{B}(e)(x)>1$.
\end{definition}

If $f_{A}$ is not soft quasi-coincident with $g_{B}$, then we write $f_{A}%
\overline{q}g_{B}$.

\begin{proposition}
\cite{Az}\label{P-c}Let $f_{A}$, $g_{B}\in FS(X,E)$, Then the followings are
true.

(1) $f_{A}\sqsubseteq g_{B}\Leftrightarrow f_{A}\overline{q}g_{B}^{c}$.

(2) $f_{A}qg_{B}\Rightarrow f_{A}\sqcap g_{B}\neq \widetilde{0}_{E}$

(3) $f_{A}\overline{q}f_{A}^{c}$.

(4) $f_{A}qg_{B}\Leftrightarrow $there exists an $e_{x}^{\alpha }\widetilde{%
\in }f_{A}$ such that $e_{x}^{\alpha }qg_{B}$.

(5) $e_{x}^{\alpha }\overline{q}f_{A}\Leftrightarrow e_{x}^{\alpha }%
\widetilde{\in }f_{A}^{c}$.

(6) $f_{A}\sqsubseteq g_{B}\Leftrightarrow $ If $e_{x}^{\alpha }qf_{A}$,
then $e_{x}^{\alpha }qg_{B}$ for all $e_{x}^{\alpha }\in FS(X,E)$.
\end{proposition}

\begin{theorem}
\label{q1}Let $f_{A}$, $g_{B}\in FS(Y,K)$, $f_{up}:FS(X,E)\rightarrow
FS(Y,K) $ be fuzzy soft mapping and $f_{A}$ is not quasi coincident with $%
g_{B}$. Then $f_{up}^{-1}(f_{A})$ is not quasi coincident with $%
f_{up}^{-1}(g_{B})$.
\end{theorem}

\begin{proof}
\begin{equation*}
\begin{array}{ll}
f_{A}\overline{q}g_{B} & \Longrightarrow \text{ For all }k\in K\text{ and }%
y\in Y\text{; }f_{A}(k)(y)+g_{B}(k)(y)\leq 1\text{ } \\ 
& \Longrightarrow \text{ For all }e\in E\text{ and }x\in X\text{; }%
f_{A}(p(e))(u(x))+g_{B}(p(e))(u(x))\leq 1\text{ } \\ 
& \Longrightarrow \text{ For all }e\in E\text{ and }x\in X\text{; }%
f_{up}^{-1}(f_{A})(e)(x)+f_{up}^{-1}(g_{B})(e)(x)\leq 1 \\ 
& \Longrightarrow \text{ }f_{up}^{-1}(f_{A})\overline{q}f_{up}^{-1}(g_{B}).%
\end{array}%
\end{equation*}
\end{proof}

\begin{definition}
(see \cite{TK, rs}) A fuzzy soft topological space is a pair $(X,{\Large %
\tau })$ where $X$ is a nonempty set and ${\Large \tau }$ is a family of
fuzzy soft sets over $X$ satisfying the following properties:

(1) $\widetilde{0}_{E},\widetilde{1}_{E}\in {\Large \tau }$

(2) If $f_{A}$, $g_{B}\in {\Large \tau }$ , then $f_{A}\sqcap g_{B}\in 
{\Large \tau }$

(3) If $f_{A}{}_{i}\in {\Large \tau }$, $\forall i\in J$, then ${\LARGE %
\sqcup }_{i\in J}f_{A}{}_{i}\in {\Large \tau }$.\newline
Then ${\Large \tau }$ is called a topology of fuzzy soft sets on $X$. Every
member of ${\Large \tau }$ is called fuzzy soft open. $g_{B}$ is called
fuzzy soft closed in $(X,{\Large \tau })$ if $(g_{B})^{c}\in {\Large \tau }$
.
\end{definition}

\begin{definition}
\cite{Az}A fuzzy soft set $f_{A}$ in $FS(X,E)$ is called Q-neighborhood
(briefly, Q-nbd) of $g_{B}$ if and only if there exists a fuzzy soft open
set $h_{C}$ in ${\Large \tau }$ such that $g_{B}qh_{C}\sqsubseteq f_{A}$.
\end{definition}

\begin{theorem}
\cite{Az}\label{t-cl}Let $e_{x}^{\alpha }$, $f_{A}\in FS(X,E)$. Then $%
e_{x}^{\alpha }\widetilde{\in }\overline{f_{A}}$ if and only if each Q-nbd
of $e_{x}^{\alpha }$ is soft quasi-coincident with $f_{A}$.
\end{theorem}

\begin{definition}
\cite{TK}A fuzzy soft set $g_{B}$ in a fuzzy soft topological space $(X,%
{\Large \tau })$ is called a fuzzy soft neighborhood (briefly: nbd) of the
fuzzy soft set $f_{A}$ if there exists a fuzzy soft open set $h_{C}$ such
that $f_{A}\sqsubseteq h_{C}\sqsubseteq g_{B}$.
\end{definition}

\begin{definition}
\cite{Az}Let $(X,{\Large \tau }_{1})$ and $(Y,{\Large \tau }_{2})$ be two
fuzzy soft topological spaces. A fuzzy soft mapping $f_{up}:(X,{\Large \tau }%
_{1})\rightarrow (Y,{\Large \tau }_{2})$ is called fuzzy soft continuous if $%
f_{up}^{-1}(g_{B})\in {\Large \tau }_{1}$ for all $g_{B}\in {\Large \tau }%
_{2}$.
\end{definition}

\begin{theorem}
\cite{Az}\label{cn}Let $(X,{\Large \tau }_{1})$ and $(Y.{\Large \tau }_{2})$
be fuzzy soft topological spaces. For a function $f_{up}:FS(X,E)%
\longrightarrow FS(Y,K)$, the following statements are equivalent:

(a)$f_{up}$ is fuzzy soft continuous;

(b) for each fuzzy soft set $f_{A}$ in $FS(X,E)$, the inverse image of every
nbd of $f_{up}(f_{A})$ is a nbd of $f_{A}$;

(c) for each soft set $f_{A}$ in $FS(X,E)$ and each nbd $h_{C}$ of $%
f_{up}(f_{A})$, there is a nbd $g_{B}$ of $f_{A}$ such that$%
f_{up}(g_{B})\sqsubseteq h_{C}$;
\end{theorem}

\begin{theorem}
\cite{Az}\label{cq}A mapping $f_{up}:(X,E)\rightarrow (Y,K)$ is continuous
if and only if corresponding fuzzy soft open $q$-nbd $g_{B}$ of $%
k_{y}^{\alpha }$ in $FS(Y,K)$ there exists a fuzzy soft open $q$-nbd $f_{A}$
of $e_{x}^{\alpha }$ in $FS(X,E)$ such that $f_{up}(f_{A})\sqsubseteq g_{B}$%
, where $f_{up}(e_{x}^{\alpha })=k_{y}^{\alpha }$
\end{theorem}

\section{quasi seperation axioms}

\begin{definition}
\cite{md}A fuzzy soft topological space $(X,\tau )$ is said to be a fuzzy
soft $T_{0}$- space if for every pair of disjoint fuzzy soft sets $e_{h_{A}}$%
, $e_{g_{B}}$, $\exists $ a fuzzy soft open set containing one but not the
other.
\end{definition}

\begin{definition}
\cite{md}A fuzzy soft topological space $(X,\tau )$ is said to be a fuzzy
soft $T_{1}$- space if for distinct pair of fuzzy soft sets $e_{g_{A}}$, $%
e_{k_{A}}$ of $f_{A}$, $\exists $ fuzzy soft open sets $s_{A}$ and $h_{A}$
such that $e_{g_{A}}\in s_{A}$ and $e_{g_{A}}\notin h_{A}$;$e_{k_{A}}\in
h_{A}$ and $e_{k_{A}}\notin s_{A}$.
\end{definition}

\begin{definition}
\cite{md}A fuzzy soft topological space $(X,\tau )$ is said to be a fuzzy
soft $T_{2}$- space if for distinct pair of fuzzy soft sets $e_{g_{A}}$, $%
e_{k_{A}}$ of $f_{A}$, $\exists $ disjoint fuzzy soft open sets $s_{A}$ and $%
h_{A}$ such that $e_{g_{A}}\in s_{A}$ and $e_{g_{A}}\notin h_{A}$;$%
e_{k_{A}}\in h_{A}$ and $e_{k_{A}}\notin s_{A}$.
\end{definition}

\begin{definition}
\cite{md}A fuzzy soft topological space $(X,\tau )$ is said to be a fuzzy
soft regular space if for every fuzzy soft set $e_{h_{A}}$ and fuzzy soft
closed set $k_{A}$ not containing $e_{h_{A}}$, $\exists $ disjoint fuzzy
soft open sets $g_{1_{A}}$, $g_{2_{A}}$ such that $e_{h_{A}}\in g_{1_{A}}$ $%
k_{A}\sqsubseteq g_{2_{A}}$.
\end{definition}

A fuzzy soft regular $T_{1}$- space is called a fuzzy soft $T_{3}$- space,

\begin{definition}
\cite{md}A fuzzy soft topological space $(X,\tau )$ is said to be a fuzzy
soft normal space if for every pair of disjoint fuzzy soft closed sets $%
h_{A} $ and $k_{A}$, $\exists $ disjoint fuzzy soft open sets $g_{1_{A}}$, $%
g_{2_{A}}$ such that $h_{A}\sqsubseteq g_{1_{A}}$ and $k_{A}\sqsubseteq
g_{2_{A}}$.
\end{definition}

A fuzzy soft normal $T_{1}$- space is called a fuzzy soft $T_{4}$- space.

\begin{definition}
Let $(X,{\Large \tau })$ be a fuzzy soft topological space and $%
(e_{1})_{x}^{\alpha }$, $(e_{2})_{y}^{\beta }\in FS(X,E)$. If there exists
fuzzy soft open sets $f_{A}$ and $g_{B}$ such that

(a) When $e_{1}\neq e_{2}$ or $x\neq y$, $f_{A}\in N((e_{1})_{x}^{\alpha })$%
, $f_{A}\overline{q}(e_{2})_{y}^{\beta }$ or $g_{B}\in N((e_{2})_{y}^{\beta
})$, $g_{B}\overline{q}(e_{1})_{x}^{\alpha }$.

(b) When $e_{1}=e_{2}$, $x=y$ and $\alpha <\beta $ (say), $f_{A}\in
N_{q}((e_{2})_{y}^{\beta })$ such that $(e_{1})_{x}^{\alpha }\overline{q}%
f_{A}$.\newline
Then $(X,{\Large \tau })$ is called a fuzzy soft $q$-$T_{0}$ space.
\end{definition}

\begin{theorem}
\label{t0-1}Let $(X,{\Large \tau })$ be a fuzzy soft topological space and $%
(X,{\Large \tau })$ fuzzy soft $q$-$T_{0}$. Then $(X,{\Large \tau })$ is
fuzzy soft $T_{0}$.
\end{theorem}

\begin{proof}
Let $(X,{\Large \tau })$ be a fuzzy soft topological space and $(X,{\Large %
\tau })$ fuzzy soft $q$-$T_{0}$. Suppose that $(X,{\Large \tau })$ is not $%
T_{0}$. Then there exist disjoint fuzzy soft sets $e_{fA}$, $e_{gB}$ such
that for every fuzzy soft open set $h_{C}$ which is containing $e_{fA}$, $%
e_{gB}$ is fuzzy soft subset $h_{C}$. Since $e_{fA}$ and $e_{gB}$ disjoint
fuzzy soft sets, there exists $(e_{1})_{x}^{\alpha }\widetilde{\in }e_{fA}$
and $(e_{2})_{y}^{\beta }\widetilde{\in }e_{gB}$ such that $e_{1}\neq e_{2}$
or $x\neq y$. Since $(e_{1})_{x}^{\alpha }\widetilde{\in }e_{fA}\sqsubset
h_{C}$ and $(e_{2})_{y}^{\beta }\widetilde{\in }e_{gB}\sqsubset h_{C}$, $%
h_{C}\in N((e_{1})_{x}^{\alpha })$ and $\beta \leq h_{C}(y)(e_{2})$. Now;%
\newline
Case I. When $\beta >0,5$, then $h_{C}(y)(e_{2})+\beta >1$. Therefore, we
have $h_{C}q(e_{1})_{x}^{\alpha }$. This is contradiction.\newline
Case II. When $\beta \leq 0,5$, if we choose $\theta >1-\beta $, then $%
h_{C}\sqcup (e_{2})_{y}^{\theta }\in N((e_{1})_{x}^{\alpha })$ and $%
(h_{C}\sqcup (e_{2})_{y}^{\theta })(y)(e_{2})+\beta >1$. Therefore, we have $%
(h_{C}\sqcup (e_{2})_{y}^{\theta })q(e_{2})_{y}^{\beta }$. This is
contradiction.
\end{proof}

\begin{theorem}
$(X,{\Large \tau })$ fuzzy soft topological space is fuzzy soft $q$-$T_{0}$
if and only if for every pair of distinct fuzzy soft points $%
(e_{1})_{x}^{\alpha }$ and $(e_{2})_{y}^{\beta }$, $(e_{1})_{x}^{\alpha }%
\widetilde{\notin }\overline{(e_{2})_{y}^{\beta }}$ or $(e_{2})_{y}^{\beta }%
\widetilde{\notin }\overline{(e_{1})_{x}^{\alpha }}$.
\end{theorem}

\begin{proof}
Let $(X,{\Large \tau })$ be fuzzy soft $q$-$T_{0}$, $(e_{1})_{x}^{\alpha }$
and $(e_{2})_{y}^{\beta }$ be two distinct fuzzy soft points in $FS(X,E)$.

Case I. When $e_{1}\neq e_{2}$ or $x\neq y$, $(e_{1})_{x}^{1}$ has an fuzzy
soft nbd $f_{A}$ such that $f_{A}\overline{q}(e_{2})_{y}^{\beta }$ or $%
(e_{2})_{y}^{1}$ has an fuzzy soft nbd $g_{B}$ such that $g_{B}\overline{q}%
(e_{1})_{x}^{\alpha }$. Suppose $(e_{1})_{x}^{1}$ has an fuzzy soft nbd $%
f_{A}$ such that $f_{A}\overline{q}(e_{2})_{y}^{\beta }$. Then $f_{A}$ is a
fuzzy soft $q$-nbd of $(e_{1})_{x}^{\alpha }$ and $f_{A}\overline{q}%
(e_{2})_{y}^{\beta }$. Hence $(e_{1})_{x}^{\alpha }\widetilde{\notin }%
\overline{(e_{2})_{y}^{\beta }}$.

Case II. When $e_{1}=e_{2}$, $x=y$ and $\alpha <\beta $ (say), then $%
(e_{2})_{y}^{\beta }$ has a \ $q$-nbd which is not quasi coincident with $%
(e_{1})_{x}^{\alpha }$ and so in this case also $(e_{2})_{y}^{\beta }%
\widetilde{\notin }\overline{(e_{1})_{x}^{\alpha }}$.

Conversely, let $(e_{1})_{x}^{\alpha }$ and $(e_{2})_{y}^{\beta }$ be two
distinct fuzzy soft points in $FS(X,E)$. We suppose without loss of
generality, that $(e_{1})_{x}^{\alpha }\widetilde{\notin }\overline{%
(e_{2})_{y}^{\beta }}$. When $e_{1}\neq e_{2}$ or $x\neq y$, since $%
(e_{1})_{x}^{\alpha }\widetilde{\notin }\overline{(e_{2})_{y}^{\beta }}$ for
all $\alpha \in (0,1]$, $(e_{2})_{y}^{\beta }=0$ and hence $(\overline{%
(e_{2})_{y}^{\beta }})^{C}(e_{1})(x))=1$. Then $(\overline{%
(e_{2})_{y}^{\beta }})^{C}$ is an fuzzy soft nbd of $(e_{1})_{x}^{\alpha }$
such that $(\overline{(e_{2})_{y}^{\beta }})^{C}\overline{q}%
(e_{2})_{y}^{\beta }$. Also, in case $e_{1}=e_{2}$ and $x=y$ we\ must have $%
\alpha >\beta $ and\ then $(e_{1})_{x}^{\alpha }$ has a fuzzy soft $q$-nbd
which is not $q$-coincident with $(e_{2})_{y}^{\beta }$.
\end{proof}

\begin{definition}
Let $(X,{\Large \tau })$ be a fuzzy soft topological space and $%
(e_{1})_{x}^{\alpha }$, $(e_{2})_{y}^{\beta }\in FS(X,E)$. If there exists
fuzzy soft open sets $f_{A}$ and $g_{B}$ such that

(a) When $e_{1}\neq e_{2}$ or $x\neq y$, $f_{A}\in N((e_{1})_{x}^{\alpha })$%
, $f_{A}\overline{q}(e_{2})_{y}^{\beta }$ and $g_{B}\in N((e_{2})_{y}^{\beta
})$, $g_{B}\overline{q}(e_{1})_{x}^{\alpha }$.

(b) When $e_{1}=e_{2}$, $x=y$ and $\alpha <\beta $ (say), $f_{A}\in
N_{q}((e_{2})_{y}^{\beta })$ such that $(e_{1})_{x}^{\alpha }\overline{q}%
f_{A}$.\newline

Then $(X,{\Large \tau })$ is called a fuzzy soft $q$-$T_{1}$ space.
\end{definition}

\begin{theorem}
Let $(X,{\Large \tau })$ be a fuzzy soft topological space.and $(X,{\Large %
\tau })$ fuzzy soft, $q$-$T_{1}$. Then $(X,{\Large \tau })$ is fuzzy soft $%
T_{1}$.
\end{theorem}

\begin{proof}
The proof is similar with the proof of Theorem \ref{t0-1}.
\end{proof}

\begin{theorem}
$(X,{\Large \tau })$ fuzzy soft topological space is $q$-$T_{1}$ if and only
if each $(e_{1})_{x}^{\alpha }\in FS(X,E)$ is a fuzzy soft closed set.
\end{theorem}

\begin{proof}
Suppose that for each $(e_{1})_{x}^{\alpha }\in FS(X,E)$ is a fuzzy soft
closed set. Then $((e_{1})_{x}^{\alpha })^{c}$ is a fuzzy soft open set. Let 
$(e_{1})_{x}^{\alpha }$, $(e_{2})_{y}^{\beta }\in FS(X,E)$.

Case I. When $e_{1}\neq e_{2}$ or $x\neq y$, for $(e_{1})_{x}^{\alpha }\in
FS(X,E)$, $((e_{1})_{x}^{\alpha })^{c}$ is a fuzzy soft open set such that $%
((e_{1})_{x}^{\alpha })^{c}\in N((e_{2})_{y}^{\beta })$ and $%
(e_{1})_{x}^{\alpha }\overline{q}((e_{1})_{x}^{\alpha })^{c}$. Similarly $%
((e_{2})_{y}^{\beta })^{c}$ is a fuzzy soft open set such that $%
((e_{2})_{y}^{\beta })^{c}\in N((e_{1})_{x}^{\alpha })$ and $%
(e_{2})_{y}^{\beta }\overline{q}((e_{2})_{y}^{\beta })^{c}$.

Case II. When $e_{1}=e_{2}$, $x=y$ and $\alpha <\beta $ (say), then $%
(e_{2})_{y}^{\beta }$ has a \ $q$-nbd $((e_{1})_{x}^{\alpha })^{c}$ which is
not quasi coincident with $(e_{1})_{x}^{\alpha }$.Thus $(X,{\Large \tau })$
is a fuzzy soft $T_{1}$ space,

Conversely, Let $(X,{\Large \tau })$ fuzzy soft topological space is $q$-$%
T_{1}$. Suppose that each fuzzy soft point $(e_{1})_{x}^{\alpha }$ is not
fuzzy closed set in ${\Large \tau }$. Then $(e_{1})_{x}^{\alpha }\neq 
\overline{(e_{1})_{x}^{\alpha }}$ and there exists $(e_{2})_{y}^{\beta }%
\widetilde{\in }\overline{(e_{1})_{x}^{\alpha }}$ such that $%
(e_{1})_{x}^{\alpha }\neq (e_{2})_{y}^{\beta }$.

Case I. When $e_{1}\neq e_{2}$ or $x\neq y$, Suppose that $\beta \leq 0,5$.
Since $(e_{2})_{y}^{\beta }\widetilde{\in }\overline{(e_{1})_{x}^{\alpha }}$%
, by theorem \ref{t-cl} for each $f_{A}\in N_{q}((e_{2})_{y}^{\beta })$, $%
f_{A}q(e_{1})_{x}^{\alpha }$. Then there exists fuzzy soft open set $h_{C}$
such that $(e_{2})_{y}^{\beta }qh_{C}\sqsubset f_{A}$. Hence $%
h_{C}(e_{2})(y)+\beta >1$ and $h_{C}(e_{2})(y)>1-\beta $. Since $%
(e_{2})_{y}^{\beta }\sqsubset (e_{2})_{y}^{1-\beta }\sqsubseteq
h_{C}\sqsubset f_{A}$, we have fuzzy soft nbd $f_{A}$ of $(e_{2})_{y}^{\beta
}$ such that $f_{A}q(e_{1})_{x}^{\alpha }$. This is contradiction. If $\beta
>0,5$, we choose $1-\beta $ the proof can be done as above.

Case II. When $e_{1}=e_{2}$, $x=y$ and $\alpha <\beta $ (say), Since $%
(e_{2})_{y}^{\beta }\widetilde{\in }\overline{(e_{1})_{x}^{\alpha }}$, by
theorem \ref{t-cl} for each $f_{A}\in N_{q}((e_{2})_{y}^{\beta })$, $%
f_{A}q(e_{1})_{x}^{\alpha }$. This is contradiction.
\end{proof}

\begin{definition}
Let $(X,{\Large \tau })$ be a fuzzy soft topological space and $%
(e_{1})_{x}^{\alpha }$, $(e_{2})_{y}^{\beta }\in FS(X,E)$. If there exists
fuzzy soft open sets $f_{A}$ and $g_{B}$ such that

(a) When $e_{1}\neq e_{2}$ or $x\neq y$, $f_{A}\in N((e_{1})_{x}^{\alpha })$%
, $g_{B}\in N((e_{2})_{y}^{\beta })$ such that $f_{A}\overline{q}g_{B}$.

(b) When $e_{1}=e_{2}$, $x=y$ and $\alpha <\beta $ (say), $f_{A}\in
N((e_{1})_{x}^{\alpha })$, $g_{B}\in N_{q}((e_{2})_{y}^{\beta })$ such that $%
f_{A}\overline{q}g_{B}$.\newline

Then $(X,{\Large \tau })$ is called a fuzzy soft $q$-$T_{2}$ space.
\end{definition}

\begin{theorem}
Let $(X,{\Large \tau })$ be a fuzzy soft topological space.and $(X,{\Large %
\tau })$ fuzzy soft, $q$-$T_{2}$. Then $(X,{\Large \tau })$ is $T_{2}$.
\end{theorem}

\begin{proof}
The proof is similar with the proof of Theorem \ref{t0-1}.
\end{proof}

\begin{remark}
Obviously, Fuzzy soft $q$-$T_{2}\Longrightarrow $Fuzzy soft $q$-$%
T_{1}\Longrightarrow $Fuzzy soft $q$-$T_{0}$.
\end{remark}

\begin{theorem}
$(X,{\Large \tau })$ be a fuzzy soft $q$-$T_{2}$ if and only if for every $%
e_{x}^{\alpha }={\LARGE \sqcap }\{\overline{f_{A}}:f_{A}$ is a fuzzy soft
nbd $e_{x}^{\alpha }\}$.
\end{theorem}

\begin{proof}
Let $(X,{\Large \tau })$ be a fuzzy soft $q$-$T_{2}$. $(e_{1})_{x}^{\alpha }$
and $(e_{2})_{y}^{\beta }$ fuzzy soft points in $FS(X,E)$ such that $%
(e_{2})_{y}^{\beta }\widetilde{\neq }(e_{1})_{x}^{\alpha }$. If $e_{1}\neq
e_{2}$ or $x\neq y$, then there are fuzzy soft open sets $f_{A}$ and $g_{B}$
containing $(e_{2})_{y}^{1}$ and $(e_{1})_{x}^{\alpha }$ respectively such
that $f_{A}\overline{q}g_{B}$. Then $g_{B}$ is a fuzzy soft open nbd of $%
(e_{1})_{x}^{\alpha }$ and $f_{A}$ is fuzzy soft $q$-nbd of $%
(e_{2})_{y}^{\beta }$ such that $f_{A}\overline{q}g_{B}$. Hence $%
(e_{2})_{y}^{\beta }\widetilde{\notin }\overline{g_{B}}$. If $e_{1}=e_{2}$, $%
x=y$, then $\beta >\alpha $ and hence there are a fuzzy soft $q$-nbd $f_{A}$
of $(e_{2})_{y}^{\beta }$ and a fuzzy soft nbd $g_{B}$ of $%
(e_{1})_{x}^{\alpha }$ such that $f_{A}\overline{q}g_{B}$. Then $%
(e_{2})_{y}^{\beta }\widetilde{\notin }\overline{g_{B}}$.

Conversely, let $(e_{1})_{x}^{\alpha }$ and $(e_{2})_{y}^{\beta }$ be two
distinct fuzzy soft points in $FS(X,E)$.

Case I. Let $e_{1}\neq e_{2}$ or $x\neq y$. We first suppose that at least
one of $\alpha $ and $\beta $ is less than $1$, say $0<\alpha <1$. The there
exists a positive real number $\lambda $ with $0<\alpha +\lambda <1$. By
hypothesis, there exists a fuzzy soft open nbd $f_{A}$ of $%
(e_{2})_{y}^{\beta }$ such that $(e_{1})_{x}^{\lambda }\widetilde{\notin }%
\overline{f_{A}}$. Then $(e_{1})_{x}^{\lambda }$ has a fuzzy soft open $q$%
-nbd $g_{B}$ such that $g_{B}\overline{q}f_{A}$. Now, $\lambda
+g_{B}(e_{1})(x)>1$ so that $g_{B}(e_{1})(x)>1-\lambda >\alpha $ and hence $%
g_{B}$ is a fuzzy soft nbd of $(e_{1})_{x}^{\alpha }$ such that $f_{A}%
\overline{q}g_{B}$, where $f_{A}$ is a fuzzy soft nbd of $(e_{2})_{y}^{\beta
}$.

In case $\alpha =\beta =1$, by hypothesis there is a fuzzy soft nbd $f_{A}$
of $(e_{1})_{x}^{1}$ such that $\overline{f_{A}}(e_{2})(y)=0$. Then $g_{B}=(%
\overline{f_{A}})^{C}$ is a fuzzy soft nbd of $(e_{2})_{y}^{1}$ such that $%
f_{A}\overline{q}g_{B}$.

Case II. Let $e_{1}=e_{2}$,$\ x=y$ and $\alpha <\beta $ (say), then there is
a fuzzy soft nbd $f_{A}$ of $(e_{1})_{x}^{\alpha }$ such that $%
(e_{2})_{y}^{\beta }\widetilde{\notin }\overline{f_{A}}$. Consequently,
there exists a fuzzy soft $q$- nbd $g_{B}$ of $(e_{2})_{y}^{\beta }$ such
that $f_{A}\overline{q}g_{B}$. Then $(X,{\Large \tau })$ fuzzy soft $q$-$%
T_{2}$.
\end{proof}

\begin{theorem}
Let $(X,{\Large \tau }_{1})$ be a fuzzy soft topological space, $(Y,{\Large %
\tau }_{2})$ be a fuzzy soft $q$-$T_{2\text{ }}$topological space and $%
f_{\varphi \psi }:FS(X,E)\rightarrow FS(Y,K)$ be injective, continuous fuzzy
soft mapping. Then $(X,{\Large \tau }_{1})$ is a fuzzy soft $q$-$T_{2}$
space.
\end{theorem}

\begin{proof}
Let $(Y,{\Large \tau }_{2})$ be a fuzzy soft $q$-$T_{2\text{ }}$topological
space and $f_{pu}:FS(X,E)\rightarrow FS(Y,K)$ be injective, continuous fuzzy
soft mapping. Let $(e_{1})_{x}^{\alpha }$, $(e_{2})_{y}^{\beta }\in FS(X,E)$,

Case I. When $e_{1}\neq e_{2}$ or $x\neq y$, then $f_{up}((e_{1})_{x}^{%
\alpha })\neq f_{up}((e_{2})_{y}^{\beta })$. Then $(Y,{\Large \tau }_{2})$
be a fuzzy soft $q$-$T_{2\text{ }}$topological space, $f_{up}((e_{1})_{x}^{%
\alpha })$, $f_{up}((e_{2})_{y}^{\beta })$ have fuzzy soft open nbds $f_{A}$%
, $g_{B}$ respctively such that $f_{A}\overline{q}g_{B}$. Then by teorem \ref%
{cq}, \ref{q1} $f_{up}^{-1}(f_{A})$ and $f_{up}^{-1}(g_{B})$ are fuzzy soft
nbds of $(e_{1})_{x}^{\alpha }$, $(e_{2})_{y}^{\beta }$ respectively such
that $f_{up}^{-1}(f_{A})\overline{q}f_{up}^{-1}(g_{B})$.

Case II. When $e_{1}=e_{2}$, $x=y$ and $\alpha <\beta $ (say), then $%
f_{up}((e_{1})_{x}^{\alpha })\neq f_{up}((e_{2})_{y}^{\beta })$. Then $(Y,%
{\Large \tau }_{2})$ be a fuzzy soft $q$-$T_{2\text{ }}$topological space, $%
f_{A}\in N(f_{up}((e_{1})_{x}^{\alpha }))$, $g_{B}\in
N_{q}(f_{up}((e_{2})_{y}^{\beta }))$ such that $f_{A}\overline{q}g_{B}$.
Then by teorem \ref{cq}, \ref{q1} $f_{up}^{-1}(f_{A})\in
N((e_{1})_{x}^{\alpha })$ and $f_{up}^{-1}(g_{B})\in
N_{q}((e_{2})_{y}^{\beta })$ such that $f_{up}^{-1}(f_{A})\overline{q}%
f_{up}^{-1}(g_{B})$. Then $(X,{\Large \tau }_{1})$ is a fuzzy soft $q$-$%
T_{2} $ space.
\end{proof}

\begin{definition}
A fuzzy soft topological space $(X,{\Large \tau })$ is fuzzy soft $q$%
-regular if and only if for any fuzzy soft closed set $h_{C}$ in $FS(X,E)$
and any fuzzy soft point $e_{x}^{\alpha }$ in $FS(X,E)$ such that $%
e_{x}^{\alpha }\widetilde{\notin }h_{C}$.

Case I. When $h_{C}(e)(x)=0$, there are fuzzy soft open sets $f_{A}$ and $%
g_{B}$ such that $e_{x}^{\alpha }\widetilde{\in }f_{A}$, $h_{C}\sqsubseteq
g_{B}$ and $f_{A}\overline{q}g_{B}$.

Case II. When $h_{C}(e)(x)\neq 0$, there are fuzzy soft open sets $f_{A}$
and $g_{B}$ such that $e_{x}^{\alpha }qf_{A}$, $h_{C}\sqsubseteq g_{B}$ and $%
f_{A}\overline{q}g_{B}$.

A fuzzy soft $q$-$T_{1}$, fuzzy soft $q$-regular fuzzy soft topological
space is a fuzzy soft $q$-$T_{3}$.
\end{definition}

\begin{theorem}
A fuzzy soft topological space $(X,{\Large \tau })$ is fuzzy soft $q$%
-regular if and only if for a fuzzy soft point $e_{x}^{\alpha }$ and any
fuzzy soft open $g_{B}$ in $FS(X,E)$ such that $e_{x}^{\alpha }qg_{B}$ there
is an fuzzy soft open set $\ f_{A}$ such that $e_{x}^{\alpha }qf_{A}$ and $%
\overline{f_{A}}\sqsubseteq g_{B}$.
\end{theorem}

\begin{proof}
Let $(X,{\Large \tau })$ be fuzzy soft $q$-regular. On the otherhand, a
fuzzy soft point $e_{x}^{\alpha }$ and fuzzy soft open $g_{B}$ in $FS(X,E)$
are given such that $e_{x}^{\alpha }qg_{B}$. Then $g_{B}^{c}$ is fuzzy soft
closed set and $e_{x}^{\alpha }\widetilde{\notin }g_{B}^{c}$.

Case I. When $g_{B}^{c}(e)(x)=0$, since $(X,\tau )$ $q-$regular, there exist
fuzzy soft open $f_{A}$ and $h_{C}$ such that $e_{x}^{1}\widetilde{\in }%
f_{A} $, $g_{B}^{c}\sqsubseteq h_{C}$ and $f_{A}\overline{q}h_{C}$. Then $%
h_{C}^{c}\sqsubseteq g_{B}$ and $f_{A}\sqsubseteq h_{C}^{c}$. Since $%
\overline{f_{A}}$ is smallest fuzzy soft closed set which is containing $%
f_{A}$, we have $f_{A}\sqsubseteq \overline{f_{A}}\sqsubseteq
h_{C}^{c}\sqsubseteq g_{B}$ and $e_{x}^{\alpha }qf_{A}$.

Case II. when $g_{B}^{c}(e)(x)\neq 0$, there exist fuzzy soft open $f_{A}$
and $h_{C}$ such that $e_{x}^{\alpha }qf_{A}$, $g_{B}^{c}\sqsubseteq h_{C}$
and $f_{A}\overline{q}h_{C}$. Since $\overline{f_{A}}$ is smallest fuzzy
soft closed set which is containing $f_{A}$, we have. $f_{A}\sqsubseteq 
\overline{f_{A}}\sqsubseteq h_{C}^{c}\sqsubseteq g_{B}$ and $e_{x}^{\alpha
}qf_{A}$.

Conversely, let any fuzzy soft closed set $h_{C}$ be in $FS(X,E)$ and any
fuzzy soft point $e_{x}^{\alpha }$ be in $FS(X,E)$ such that $e_{x}^{\alpha }%
\widetilde{\notin }h_{C}$. Then $h_{C}^{c}$ is fuzzy soft open and $%
e_{x}^{\alpha }qh_{C}^{c}$.

Case I. When $h_{C}(e)(x)=0$, if $\alpha \leq 0,5$, then there exists fuzzy
soft open$\ f_{A}$ such that $e_{x}^{\alpha }\widetilde{\in }f_{A}$ and $%
\overline{f_{A}}\sqsubseteq h_{C}^{c}$. Therefore, we have $e_{x}^{\alpha }%
\widetilde{\in }f_{A}$, $h_{C}\sqsubseteq (\overline{f_{A})}^{c}$ and $f_{A}%
\overline{q}(\overline{f_{A})}^{c}$. if $\alpha >0,5$, then there exists
fuzzy soft open$\ f_{A}$ such that $e_{x}^{1-\alpha }\widetilde{\in }f_{A}$
and $\overline{f_{A}}\sqsubseteq h_{C}^{c}$. Therefore, we have $%
e_{x}^{\alpha }\widetilde{\in }f_{A}$, $h_{C}\sqsubseteq (\overline{f_{A})}%
^{c}$ and $f_{A}\overline{q}(\overline{f_{A})}^{c}$.

Case II. When $h_{C}(e)(x)\neq 0$, there exists fuzzy soft open $f_{A}$ such
that $e_{x}^{\alpha }qf_{A}$ and $\overline{f_{A}}\sqsubseteq h_{C}^{c}$.
So, we have $e_{x}^{\alpha }qf_{A}$, $h_{C}\sqsubseteq (\overline{f_{A})}^{c}
$ and $f_{A}\overline{q}(\overline{f_{A})}^{c}$. As a result, $(X,{\Large %
\tau })$ is a fuzzy soft $q$-regular.
\end{proof}

\textbf{Compliance with Ethical Standards:}

\textbf{Conflict of Interest: }I (Serkan ATMACA) declare that I have no
conflict of interest. 

\textbf{Ethical approval: }This article does not contain any studies with
human participants or animals performed by me.

\end{document}